\documentclass[final,leqno]{siamltex704}
\usepackage{eepic}
\usepackage{tikz}
\usepackage{amsmath}
\usepackage{amssymb}
\usepackage{graphicx}
\usepackage[notcite,notref]{showkeys}

\newtheorem{assumptionletter}{Assumption}
\newcommand{\bq}{\begin{equation}}
\newcommand{\eq}{\end{equation}}

\renewcommand{\ldots}{\dotsc}
\def\T{{\mathcal T}}

\def\bbb{{\bf b}}

\def\bn{{\bf n}}
\def\bq{{\bf q}}

\def\3bar{{|\hspace{-.02in}|\hspace{-.02in}|}}

\title{Maximum Principles for $P1$-Conforming Finite Element Approximations of
Quasi-Linear Second Order Elliptic Equations}
\author{Junping Wang\thanks{Division of Mathematical Sciences, National
Science Foundation, Arlington, VA 22230 (jwang@\break nsf.gov). The
research of Wang was supported by the NSF IR/D program, while
working at the Foundation. However, any opinion, finding, and
conclusions or recommendations expressed in this material are those
of the author and do not necessarily reflect the views of the
National Science Foundation.} \and Ran Zhang \thanks{Department of
Mathematics, Jilin University, Changchun, China
(zhangran@mail.jlu.edu.cn). The research of Zhang was supported in
part by China Natural National Science Foundation.}}
\begin{document}
\maketitle

\begin{abstract}
This paper derives some discrete maximum principles for $P1$-conforming
finite element approximations for quasi-linear second order elliptic
equations. The results are extensions of the classical maximum
principles in the theory of partial differential equations to finite
element methods. The mathematical tools are based on the
variational approach that was commonly used in the classical PDE theory.
The discrete maximum principles are established by assuming a property
on the discrete variational form that is of global nature. In particular,
the assumption on the variational form is verified when the finite element
partition satisfies some angle conditions. For the general quasi-linear
elliptic equation, these angle conditions indicate that each triangle or
tetrahedron needs to be $\mathcal{O}(h^\alpha)$-acute in
the sense that each angle $\alpha_{ij}$ (for triangle) or interior
dihedral angle $\alpha_{ij}$ (for tetrahedron) must satisfy
$\alpha_{ij}\le \pi/2-\gamma h^\alpha$ for some $\alpha\ge 0$ and
$\gamma>0$. For the Poisson problem where the differential operator
is given by Laplacian, the angle requirement is the same as the
existing ones: either all the triangles are non-obtuse or each
interior edge is non-negative. It should be pointed out that the
analytical tools used in this paper are based on the powerful De
Giorgi's iterative method that has played important roles in the
theory of partial differential equations. The mathematical analysis
itself is of independent interest in the finite element analysis.
\end{abstract}


\pagestyle{myheadings}

\section{Introduction}

In this paper we are concerned with maximum principles for $P1$
conforming finite element solutions for quasi-linear second order
elliptic equations. The continuous problem seeks an unknown function
with appropriate regularity such that
\begin{equation}\label{FEME:eq1}
-\nabla\cdot(a(x,u,\nabla u)\nabla u) + \mathbf{b}(x,u, \nabla
u)\cdot\nabla u + c(x,u) u=f(x), \quad \mbox{in}\ \Omega,
\end{equation}
where $\Omega$ is a polygonal or polyhedral domain in
$\mathbb{R}^d\; (d=2,3)$, $a=a(x,u,\nabla u)$ is a scalar function,
$\bbb=(b_i(x,u,\nabla u))_{d\times 1}$ is a vector-valued function,
$c=c(x,u)$ is a scalar function on $\Omega$, and $\nabla u$ denotes
the gradient of the function $u=u(x)$. We shall assume that the
differential operator is strictly elliptic in $\Omega$; that is,
there exists a positive number $\lambda>0$ such that
\begin{equation}\label{ellipticity}
a(x,\eta,p)\ge \lambda,\quad\forall x\in \Omega, \eta\in \mathbb{R},
p\in \mathbb{R}^d.
\end{equation}
We also assume that the differential operator has bounded
coefficients; that is for some constants $\Lambda$ and $\nu\ge 0$ we
have
\begin{equation}\label{boundedness}
 |a(x,\eta,p)|\leq \Lambda, \quad
 \lambda^{-2}\sum |b_i(x,\eta,p)|^2 + \lambda^{-2} |c(x,\eta)|^2 \leq \nu^2,
 \end{equation}
for all $x\in \Omega, \eta\in \mathbb{R},$ and $p\in \mathbb{R}^d$.

Introduce the following form
\begin{equation}\label{bilinear-form}
\mathfrak{Q}(w;u,v):=\int_\Omega \left\{a\nabla u\cdot\nabla v +
\bbb\cdot(\nabla u) v + cuv\right\}dx,
\end{equation}
where $a=a(x,w,\nabla w)$, $\bbb=\bbb(x,w,\nabla w)$, and
$c=c(x,w)$. Let the function $f$ in (\ref{FEME:eq1}) be locally
integrable in $\Omega$. Then a weakly differentiable function $u$ is
called a weak solution of (\ref{FEME:eq1}) in $\Omega$ if
\begin{equation}\label{weak-solution}
\mathfrak{Q}(u;u,v)=F(v),\qquad \forall v\in C_0^1(\Omega),
\end{equation}
where $F(v)\equiv \int_\Omega f v dx$. For simplicity, we shall
consider solutions of (\ref{FEME:eq1}) with a non-homogeneous
Dirichlet boundary condition
\begin{equation}\label{bc-diri}
u=g, \quad \mbox{on}\  \partial\Omega,
\end{equation}
where $g\in H^{\frac12}(\partial\Omega)$ is a function defined on
the boundary of $\Omega$. Here $H^1(\Omega)$ is the Sobolev space
consisting of functions which, together with its gradient, is square
square integrable over $\Omega$. $H^{\frac12}(\partial\Omega)$ is the
trace of $H^1(\Omega)$ on the boundary of $\Omega$. The
corresponding weak form seeks $u\in H^1(\Omega)$ such that $u=g$ on
$\partial\Omega$ and
\begin{eqnarray}\label{weakform}
\mathfrak{Q}(u;u,v)=F(v),\qquad \forall v\in H_0^1(\Omega).
\end{eqnarray}

The usual maximum principle for the solution of (\ref{weakform})
(e.g., see \cite{gilbarg}) asserts that if $c(x,\eta)\ge 0$ and
$f(x)\le 0$ for all $x\in \Omega$ and $\eta\in \mathbb{R}$, then
\begin{equation}\label{mp-v1}
\sup_{x\in\Omega} u(x) \leq \sup_{x\in\partial\Omega} g_+(x),
\end{equation}
where $g_+(x)=\max(g(x),0)$ is the non-negative part of the boundary
data. Moreover, if $c=0$, then one has
\begin{equation}\label{mp-v2}
\sup_{x\in\Omega} u(x) \leq \sup_{x\in\partial\Omega} g(x).
\end{equation}
For general non-homogeneous equation (\ref{FEME:eq1}), by using the
powerful De Giorgi's iterative technique \cite{degiorgi} one can
derive the following maximum principle (see \cite{wyw} for details).

\begin{theorem}\label{pde-mp}
Let $u\in H^1(\Omega)$ be a weak solution of (\ref{FEME:eq1}) and
(\ref{bc-diri}) arising from the formula (\ref{weakform}). Let $p>2$
be any real number such that $p<+\infty$ for $d=2$ and $p<\frac{2d}{d-2}$ for $d>2$.
Assume that $f\in L^{\frac{pr}{(p-1)(r-1)}}(\Omega)$ with a real number $r$, $1\leq r < p-1$.
Assume that the coefficient functions and the solution satisfy $c-\frac12 \nabla\cdot \bbb \ge 0$
for any $x\in \Omega$. Then, there exists a constant $C=C(\Omega)$ such that
\begin{equation}\label{dmp-case1-intro-new}
\sup_{x\in \Omega} u(x) \le k_0+
C\|f\|_{L^{\frac{pr}{(p-1)(r-1)}}},
\end{equation}
where
\begin{equation}\nonumber
k_0=\left\{
\begin{array}{ll}
\displaystyle \sup_{x\in\partial\Omega} g_+(x), & \mbox{ if\  $c\ge 0$}, \\
\displaystyle \sup_{x\in\partial\Omega} g(x), & \mbox{ if \ $c\equiv 0$}.
\end{array}\right.
\end{equation}
Moreover, the dependence of $C=C(\Omega)$ is given by
$$
C(\Omega)= C 2^{\frac{p-1}{p-1-r}} |\Omega|^{\frac{p-1-r}{pr}}.
$$
\end{theorem}
Here and in what follows of this paper, $C$ denotes a generic
dimensionless constant.

\bigskip

The goal of this paper is to establish an analogy of the maximum
principles (\ref{mp-v1}), (\ref{mp-v2}), and (\ref{dmp-case1-intro-new})
for $P1$-conforming finite element approximations of (\ref{weakform}).
We will establish similar maximum principles for such finite element
approximations with an assumption on the form $\mathfrak{Q}(w;u,v)$ (see (\ref{hypothesis})
for details) that can be verified through some geometric conditions imposed on the
corresponding finite element partition. As an example, we
shall explore some geometric conditions that apply to the angles of
each element, as was commonly done in existing results on discrete
maximum principles (DMP) (see for example, \cite{ciarlet-1} and \cite{strang}).
For the general quasi-linear elliptic equation (\ref{FEME:eq1}), the
triangles or tetrahedron need to be $\mathcal{O}(h^\alpha)$-acute in
the sense that each angle (for triangular case) or interior dihedral
angle (for tetrahedral case) must satisfy $\alpha_{ij}\le
\pi/2-\gamma h^\alpha$ for some $\alpha\ge 0$ and $\gamma>0$. For
the Poisson problem where the differential operator is given by
Laplacian, the angle requirement is the same as the existing ones:
either all the triangles are non-obtuse or each interior edge is
non-negative as defined in \cite{scott}.

\medskip

The research on discrete maximum principles for finite element
solutions can be dated back to the seventies of the last century.
In \cite{ciarlet-1}, a linear second order elliptic equation was
considered, and a discrete maximum principle was established for
continuous piecewise linear finite element approximations if all
angles in the finite element triangulation are not greater than
$\pi/2$ (the so-called non-obtuse condition). In \cite{strang}, it
was noted (see page 78) that the discrete maximum principle holds
true for continuous piecewise linear finite element approximations
for the Poisson problem under the following weaker condition: for
every pair $(\alpha_1; \alpha_2)$ of angles opposite a common edge
of some given pair of adjacent triangles of the triangulation one
has $\alpha_1+\alpha_2\le \pi$. In \cite{santos}, it was shown that
the discrete maximum principle may hold true in some cases if both
angles in such a pair are greater than $\pi/2$. In \cite{hall}, the
case of rectangular meshes and bilinear finite
element approximations was considered for second order linear
elliptic equations with Dirichlet boundary conditions. The notion of
non-narrow rectangular element was introduced as a sufficient
geometric condition for a discrete maximum principle to hold. In
\cite{linqun}, a 3D nonlinear elliptic problem with Dirichlet
boundary condition was considered and the effect of quadrature rules
was taken into account. A corresponding discrete maximum principle
was derived under the condition of non-obtuseness for the underlying
tetrahedral meshes. It was further shown that the DMP may also hold
true for continuous piecewise linear finite element approximations
for elliptic problems under various weaker conditions on the
simplicial meshes used.  The acuteness assumption has been weakened
in \cite{korotov} and \cite{santos}. In particular, in certain
situations, obtuse interior angles in the simplices of the meshes
are acceptable. In \cite{karatson1}, quasi-linear elliptic equation
of second order in divergent form was considered, and corresponding
DMPs were derived for mixed (Robin-type) boundary conditions.
In \cite{schatz}, a weaker discrete maximum principle is shown to
hold under quite general conditions on the mesh (quasi-uniformity)
and arbitrary degree polynomials, namely
$$
\|u_h\|_{\infty,\Omega}\leq C \|u_h\|_{\infty,\partial\Omega},
$$
where $C>0$ is independent of the meshsize $h$. In \cite{scott},
positivity for discrete Green's function was investigated for
Poisson equations. The authors addressed the question of whether the
discrete Green's function is positive for triangular meshes allowing
sufficiently good approximation of $H^1$ functions. They gave
examples which show that in general the answer is negative. The
authors also extended the number of cases where it is known to be
positive.

\medskip

The contributions of this paper are as follows: (1) the DMP result
with general non-homogeneous quasi-linear elliptic PDE
(\ref{FEME:eq1}) is new (see Theorem \ref{thm5.2}); (2) the DMP
result, as summarized in Theorem \ref{thm5.3}, is new with the
inclusion of the first order term $\bbb(x,u,\nabla u)\cdot\nabla u$
in the PDE; and (3) the mathematical tools for deriving DMPs are new
in the finite element analysis. Our analytical tools are based on a
variational approach which are extensions of similar tools that were
used to derive maximum principles in pure theory of partial
differential equations. We envision that the new analytical tool
shall have applications to a much wider class of problems than the
existing approach based on the inversion of $M$-matrices in the DMP
analysis. In particular, we shall report some DMPs for
$P1$-nonconforming finite elements and mixed finite element
approximations for (\ref{FEME:eq1}) and (\ref{bc-diri}) in a
forthcoming paper.
\medskip

The paper is organized as follows. In Section 2, we shall review the
finite element method for (\ref{FEME:eq1}) and (\ref{bc-diri}) based
on the form (\ref{weakform-fem}). In Section 3, we shall derive two discrete
maximum principles (DMP) for $P1$-conforming finite element approximations
under an assumption to be verified in forthcoming sections. In Section 4,
we discuss the relation of shape functions with angles and interior
dihedral angles for each element (triangular or tetrahedral). Finally
in Section 5, we shall verify the assumption under which the DMPs were derived
in earlier sections by requiring some angle conditions for the underlying
finite element partition.

\section{Galerkin Finite Element Methods}\label{section2}

In the standard Galerkin method (e.g., see \cite{ci, sue}), the
trial space $H^1(\Omega)$ and the test space $H_0^1(\Omega)$ in
(\ref{weakform}) are each replaced by properly defined subspaces of
finite dimensions. The resulting solution in the subspace/subset is
called a Galerkin approximation. Galerkin finite element methods are
particular examples of the Galerkin method in which the approximating functions
(both trial and test) are given as continuous piecewise polynomials over a
prescribed finite element partition for the domain, denoted by
${\cal T}_h$.
\medskip

We consider only Galerkin finite element approximations arising from
continuous piecewise linear finite element functions -- known as
$P1$ conforming finite element methods. To this end, let ${\cal
T}_h$ be a finite element partition of the domain $\Omega$
consisting of triangles ($d=2$) or tetrahedra ($d=3$). Assume that
the partition ${\cal T}_h$ is shape regular so that the routine
inverse inequality in the finite element analysis holds true (see
\cite{ci}). Denote by $h=\max_{T\in {\cal T}_h} h_T$ the meshsize of
${\cal T}_h$ with $h_T$ being the diameter of $T$. For each $T\in
{\cal T}_h$, denote by $P_j(T)$ the set of polynomials on $T$ with
degree no more than $j$. The $P1$ conforming finite element space is
given by
\begin{equation}\label{p1conformingfem}
S_h:=\left\{ v:\ v\in H^1(\Omega),\; v|_{T}\in P_1(T), \forall T\in {\cal T}_h \right\}.
\end{equation}
Denote by $S_h^0$ the subspace of $S_h$ with
vanishing boundary values on $\partial\Omega$; i.e.,
\begin{equation}\label{p1conformingzerobdry}
S_h^0:=\left\{ v \in S_h,\; {v}|_{\partial\Omega}=0 \right\}.
\end{equation}
The corresponding Galerkin method seeks $u_h\in S_h$ such that
$u_h=I_hg$ on $\partial \Omega$ and
\begin{eqnarray}\label{weakform-fem}
\mathfrak{Q}(u_h;u_h,v)=F(v),\qquad \forall v\in S_h^0,
\end{eqnarray}
where $I_hg$ is an appropriately defined interpolation of the
Dirichlet boundary condition (\ref{bc-diri}) into continuous
piecewise linear functions on $\partial\Omega$. For example, the
standard nodal point interpolation would be acceptable if the
boundary data $u=g$ is sufficiently regular.

\medskip

Let $v\in S_h$ be any finite element function and $k$ be any real number.
We shall decompose $v-k$ into two components
\begin{equation}\label{decomposition-v}
v-k=(v-k)_+ + (v-k)_-,
\end{equation}
where $(v-k)_+$ is a finite element function in $S_h$ taken as the non-negative
part of $v-k$ at the nodal points of the finite element partition ${\cal T}_h$;
i.e., $(v-k)_+$ is defined as a function in $S_h$ such that at each nodal point $A$,
$$
(v-k)_+(A) = \left\{
\begin{array}{ll}
v(A)-k,&\qquad \mbox{if } v(A)\ge k,\\
0,&\qquad\mbox{otherwise}.
\end{array}
\right.
$$
Likewise, the function $(v-k)_-:=(v-k)-(v-k)_+$ is the non-positive
part of $v-k$ at the nodal points of ${\cal T}_h$.
\bigskip

\begin{lemma}\label{technical-lemma1}
Let $v\in S_h$ be any finite element function. Let $k$ be any real
number such that $k\ge 0$ if $c=c(x,\tau)\ge 0$ and $k$ arbitrary if
$c\equiv 0$. Then, we have
\begin{eqnarray}\label{technical.01}
\mathfrak{Q}(v;v,(v-k)_+)&\ge& \mathfrak{Q}(v;(v-k)_+,(v-k)_+) + \mathfrak{Q}(v;(v-k)_-,(v-k)_+).
\end{eqnarray}
\end{lemma}

\begin{proof}
Observe that $\mathfrak{Q}(w;u,v)$ is bilinear in terms of $u$ and $v$. Thus,
\begin{eqnarray*}\nonumber
\mathfrak{Q}(v;v, (v-k)_+) &=& \mathfrak{Q}(v;v-k, (v-k)_+) + \mathfrak{Q}(v;k, (v-k)_+)\\
&=& \mathfrak{Q}(v;v-k, (v-k)_+) + k(c, (v-k)_+).
\end{eqnarray*}
Here we have used the fact that $\mathfrak{Q}(v;k, (v-k)_+) =k(c, (v-k)_+)$.
If $c\ge 0$ and $k\ge 0$, then we obtain
\begin{eqnarray}\label{hello.88}
\mathfrak{Q}(v;v, (v-k)_+) &\ge& \mathfrak{Q}(v;v-k, (v-k)_+).
\end{eqnarray}
In the case of $c\equiv 0$, (\ref{hello.88}) clearly holds true
for any real number $k$ and the inequality can be replaced by
equality. It follows from (\ref{hello.88}) and the decomposition (\ref{decomposition-v})
that (\ref{technical.01}) holds true. This completes the proof of the lemma.
\end{proof}

\bigskip
For convenience of analysis, we shall need a discrete equivalence for the usual $L^p$ norm
$\|v\|_{L^p}$ in the finite element space $S_h$. To this end, let $v$ be any finite
element function in $S_h$. Denote by $\{v\}$ the vector
$$
\{v\}=(v(A_1),\ldots, v(A_j),\ldots, v(A_N)),
$$
where $\{A_j\}_{j=1,\cdots, N}$ is the set of nodal points of the
finite element partition $\T_h$. Denote by $\Omega_j$ the macro element associated with the
nodal point $A_j$ (i.e., $\Omega_j$ is
the union of elements $T_{ij}$ that share $A_j$ as a vertex point).
It is not hard to show that there exist constants $C_0$ and $C_1$
such that
\begin{equation}\label{normequivalence}
C_0\sum_{j=1}^N |v(A_j)|^p |\Omega_j| \leq \|v\|_{L^p}^p \leq
C_1\sum_{j=1}^N |v(A_j)|^p|\Omega_j|.
\end{equation}
For completeness, let us outline a proof for the left inequality.
For any $x\in \Omega_j$, we have
$$
v(A_j)=v(x)+(A_j-x)\cdot\nabla v.
$$
Thus,
$$
|v(A_j)|^p\leq 2^p \left(|v(x)|^p + \|(A_j-x)\|^p\ \|\nabla
v\|^p\right).
$$
Integrating over $\Omega_j$ and then using the standard inverse
inequality for the finite element function $v$ yields
$$
|v(A_j)|^p |\Omega_j| \leq C\int_{\Omega_j} |v(x)|^p dx.
$$
By summing the above over all the nodal points $A_j$ we obtain
$$
\sum_{j=1}^N |v(A_j)|^p |\Omega_j| \leq C \int_{\Omega} |v|^p dx,
$$
where we have used the fact that $\Omega_j$ overlaps with only a
fixed number of other macro-elements.

\section{Maximum Principles for $P1$ Conforming
Approximations}\label{section-dmp} The goal of this section is to
establish a maximum principle for $P1$ conforming finite element
approximations $u_h$ arising from the formula (\ref{weakform-fem}).
This shall be accomplished by using a technique known as the De
Giorgi's iterative method (\cite{degiorgi}) originally developed for
second order elliptic equations associated with maximum principles.
In its essence, the De Giorgi's iterative technique is to estimate
the set
$$
G(k):=\{x: \ x\in\Omega, u(x)\ge k\}
$$
by showing that the measure of the set $G(k)$ is zero for some
values of $k$. The center piece of the De Giorgi's iterative method
is the following technical lemma which can be proved through an
iterative argument, and hence the name of the method.

\medskip
\begin{lemma}{(\cite{degiorgi})}\label{FEME:lem1}
Let $\phi(t)$ be a non-negative monotone function on $[k_0,
+\infty)$. Assume that $\phi$ is non-increasing and satisfies
\begin{eqnarray}\label{FEME:eq1.new}
\phi(s)\le \left(\frac{M}{s-k}\right)^\alpha [\phi(k)]^\beta,
\quad \forall\ s>k\ge k_0,
\end{eqnarray}
where $\alpha>0, \beta>1$ are two fixed parameters and $M>0$ is a constant.
Then, there exists a number $\rho$ such that
$$
\phi(k_0+\rho)=0.
$$
Moreover, one has the following estimate
$$
\rho \ge M[\phi(k_0)]^{(\beta-1)/\alpha}2^{\beta/(\beta-1)}.
$$
\end{lemma}

A proof of Lemma \ref{FEME:lem1} can be found in \cite{wyw}.
Readers can also find more applications of this lemma
in the study of partial differential equations. For completeness, we outline
a proof of Lemma \ref{FEME:lem1} as follows. Let $\rho$ be a real number to be determined later, and set
$$
k_\tau=k_0+\rho-\frac{\rho}{2^\tau},\quad \tau=0,1,2,\cdots.
$$
It then follows from (\ref{FEME:eq1.new}) that the following recursive formula holds true
\begin{eqnarray}\label{FEME:eq1.new_2}
\phi(k_{\tau+1})\le \frac{M^\alpha
2^{(\tau+1)\alpha}}{\rho^\alpha}[\phi(k_\tau)]^\beta, \quad
\tau=0,1,2,\cdots
\end{eqnarray}
We claim that (\ref{FEME:eq1.new_2}) implies the following
\begin{eqnarray}\label{FEME:eq1.new_3}
\phi(k_{\tau})\le \frac{\phi(k_0)}{r^\tau}, \quad \tau=0,1,2,\cdots
\end{eqnarray}
with some real number $r>1$ to be chosen. In fact, (\ref{FEME:eq1.new_3})
can be proved by a mathematical induction. The formula (\ref{FEME:eq1.new_3})
is clearly true with any real number $r>1$ when $\tau=0$. Assume that
(\ref{FEME:eq1.new_3}) is valid for $\tau$. Now using
(\ref{FEME:eq1.new_2}) one obtains
\begin{eqnarray*}
\phi(k_{\tau+1})&\le& \frac{M^\alpha
2^{(\tau+1)\alpha}}{\rho^\alpha}[\phi(k_\tau)]^\beta
\\
&\le& \frac{\phi(k_0)}{r^{\tau+1}}\cdot \frac{M^\alpha
2^{(\tau+1)\alpha}}{\rho^\alpha
r^{\tau(\beta-1)-1}}[\phi(k_0)]^{\beta-1}.
\end{eqnarray*}
Now if we choose $r=2^{\alpha/(\beta-1)}$, then
\begin{eqnarray*}
\phi(k_{\tau+1}) \le \frac{\phi(k_0)}{r^{\tau+1}}\cdot
\frac{M^\alpha 2^{\alpha\beta/(\beta-1)}}{\rho^\alpha
}[\phi(k_0)]^{\beta-1}.
\end{eqnarray*}
From this, we see that (\ref{FEME:eq1.new_3}) is also valid for
$\tau+1$ if $\rho =
M[\phi(k_0)]^{(\beta-1)/\alpha}2^{\beta/(\beta-1)}$. Now
by taking $\tau\to +\infty$ in (\ref{FEME:eq1.new_3}), we
see that the left limit of $\phi$ at $k_0+\rho$ must be zero. This, together with
the given monotonicity of $\phi$, completes a proof for the De Giorgi Lemma.

\medskip

Let $p>2$ be any real number such that
\begin{equation}\label{sobolevp}
p<\left\{\begin{array}{ll} +\infty,\quad & d=2,
\\
\frac{2d}{d-2}, & d>2.\end{array}\right.
\end{equation}
Next, we introduce a number $k_*$ defined as follows
\begin{equation}\label{kzero}
k_*=\left\{ \begin{array}{ll}\displaystyle \sup_{x\in\partial\Omega}
\max\{I_hg(x),0\},&\quad \mbox{if $c\ge 0$},\\
\displaystyle\sup_{x\in\partial\Omega} I_hg(x),&\quad \mbox{if
$c\equiv 0$}.
\end{array}
 \right.
\end{equation}

\begin{assumptionletter}\label{assumptionA} Let the form $\mathfrak{Q}(w;u,v)$
be given by (\ref{bilinear-form}), and $u_h$ be the finite element approximation
of $u$ arising from (\ref{weakform-fem}). For any real number $k\ge k_*$, assume the
following holds true:
\begin{eqnarray}\label{hypothesis}
\mathfrak{Q}(u_h;(u_h-k)_-,(u_h-k)_+) \ge 0.
\end{eqnarray}
\end{assumptionletter}

We are now in a position to derive a maximum principle
for $P1$ conforming finite element approximations.

\begin{theorem}\label{thm5.2}
Let $u_h\in S_h$ be the $P1$-conforming finite element approximation
of (\ref{FEME:eq1}) and (\ref{bc-diri}) arising from the formula
(\ref{weakform-fem}). Denote by $I_hg$ the interpolation of the
Dirichlet boundary data (\ref{bc-diri}) that was used in the finite
element formula (\ref{weakform-fem}). Let $p$ and $r$ be real numbers satisfying (\ref{sobolevp})
and $1\leq r < p-1$. Assume that $f\in
L^{\frac{pr}{(p-1)(r-1)}}(\Omega)$ and the {\em Assumption} \ref{assumptionA} holds true. Also assume that
\begin{equation}\label{usualcondition-on-bc}
c(x,u_h) - \frac12 \nabla\cdot \bbb(x, u_h, \nabla u_h) \ge 0,\qquad \forall x\in \Omega.
\end{equation}
Then, there exists a constant $C=C(\Omega)$ such that
\begin{equation}\label{dmp-first-result}
\sup_{x\in \Omega} u_h(x) \le k_* + C\|f\|_{L^{\frac{pr}{(p-1)(r-1)}}},
\end{equation}
where $k_*$ is given by (\ref{kzero}). Moreover, the dependence of $C=C(\Omega)$ is given by
$$
C(\Omega)= C 2^{\frac{p-1}{p-1-r}} |\Omega|^{\frac{p-1-r}{pr}}.
$$
\end{theorem}
\medskip

\begin{proof} Let $k\ge k_*$ be any real number. Denote by $\varphi=(u_h-k)_+$
the positive part of $u_h-k$ at nodal points. Since $k\ge k_*$ and
$k_*$ is no smaller than the maximum value of the finite element
solution $u_h$ on $\partial\Omega$, then $\varphi$ must vanish on
the boundary of $\Omega$; i.e.,
\begin{equation}\label{ukpzero}
\varphi(x)\in S_h^0.
\end{equation}
Thus, $\varphi$ is eligible as a test function in the finite element
formulation (\ref{weakform-fem}). By taking $v=\varphi$ in
(\ref{weakform-fem}), we obtain from (\ref{technical.01}) and the Assumption \ref{assumptionA}
that
\begin{eqnarray}\nonumber
F(\varphi)&=&\mathfrak{Q}(u_h;u_h, \varphi) \\
&=&\mathfrak{Q}(u_h;u_h, (u_h-k)_+) \nonumber \\
&\ge& \mathfrak{Q}(u_h;(u_h-k)_+, (u_h-k)_+) + \mathfrak{Q}(u_h;(u_h-k)_-, (u_h-k)_+)\nonumber \\
&\ge& \mathfrak{Q}(u_h;(u_h-k)_+, (u_h-k)_+).\label{mp.18}
\end{eqnarray}
Using the notation $\varphi=(u_h-k)_+$ in (\ref{mp.18}) we obtain
\begin{eqnarray}\label{mp.19}
(a\nabla \varphi, \nabla\varphi) + (\bbb\cdot\nabla\varphi, \varphi) + (c\varphi, \varphi) \equiv
\mathfrak{Q}(u_h;\varphi, \varphi) \le F(\varphi).
\end{eqnarray}
Since the usual integration by parts implies
$$
(\bbb\cdot\nabla\varphi, \varphi) = - (\varphi, \bbb\cdot\nabla\varphi) - (\varphi, (\nabla\cdot\bbb)\varphi),
$$
then we have
$$
(\bbb\cdot\nabla\varphi, \varphi) = -\frac12 ((\nabla\cdot\bbb)\varphi, \varphi).
$$
Substituting the above into (\ref{mp.19}) yields,
$$
(a\nabla \varphi, \nabla\varphi) + \left( (c-\frac12\nabla\cdot\bbb) \varphi, \varphi\right) \le F(\varphi),
$$
which, along with the condition (\ref{usualcondition-on-bc}), leads to
$$
(a\nabla \varphi, \nabla\varphi) \le F(\varphi).
$$

Now let $G(k)$ be the subset of $\Omega$ where $\varphi>0$; i.e.,
$$
G(k)=\{T:\ T\in \mathcal{T}_h, \; \varphi>0 \mbox{ for some $x\in
T$}\}.
$$
Denote by $|G(k)|$ the Lebesgue measure of the set $G(k)$. We are
going to show that $|G(k)|=0$ for sufficiently large values of $k$.
To this end, we apply the ellipticity (\ref{ellipticity}) and the
usual H\"older inequality to (\ref{mp.19}) to obtain
\begin{eqnarray}\label{interesting}
\lambda \int_\Omega |\nabla \varphi|^2 dx \le
\|\varphi\|_{L^p(\Omega))} \|f\|_{L^q(G(k))},
\end{eqnarray}
where $p>2$ satisfies (\ref{sobolevp}) and $q$ is the conjugate of
$p$; i.e., $\frac{1}{p}+\frac{1}{q}=1$. Here the $L^q$ norm of $f$ was
taken on the support of $\varphi$ for the obvious reason. Combining the usual Sobolev
embedding with the estimate (\ref{interesting}) yields
\begin{eqnarray}\label{startingpoint}
\|\varphi\|_{L^p}^2&\le& C\|\nabla \varphi\|_{L^2}^2 \le
C\|f\|_{L^q(G(k))} \|\varphi\|_{L^p}.
\end{eqnarray}
It follows that
\begin{eqnarray*}
\|\varphi\|_{L^p}\le C\|f\|_{L^q(G(k))} \le
C\|f\|_{L^{qs}}|G(k)|^{\frac{1}{qr}},
\end{eqnarray*}
where $r\ge 1$ and $\frac{1}{r}+\frac{1}{s}=1$ are arbitrary real
numbers. The above inequality can be rewritten as
$$
\|\varphi\|_{L^p}^p\le C\|f\|_{L^{qs}}^p|G(k)|^{\frac{p}{qr}}.
$$
Now using the norm equivalence (\ref{normequivalence}) we obtain
\begin{equation}\label{hello.08}
C_0 \sum_{j=1}^N [(u_h-k)_+(A_j)]^p |\Omega_j| \leq
C\|f\|_{L^{qs}}^p|G(k)|^{\frac{p}{qr}}.
\end{equation}
It is not hard to see that $G(k)$ is the union of all the
macro-elements $\Omega_j$ so that $u_h(A_j)>k$. For any $\rho>k$, one
would have a corresponding set $G(\rho)$. Moreover, if
$\Omega_{j_0}\subset G(\rho)$, then we must have
$u_h(A_{j_0})>\rho>k$. This implies that $\Omega_{j_0}\subset G(k)$.
Therefore, we have
\begin{eqnarray*}
C_0\sum_{j=1}^N [(u_h-k)_+(A_j)]^p |\Omega_j|&\ge&
C_0\sum_{j=1,\cdots,N; u_h(A_j)> \rho} [(u_h-k)_+(A_j)]^p
|\Omega_j|\\
&\ge & C_0 \ (\rho-k)^p \sum_{j=1,\cdots,N; u_h(A_j)> \rho}
|\Omega_j|\\
&\ge& \tilde C_0 (\rho-k)^p |G(\rho)|.
\end{eqnarray*}
Substituting the above inequality into (\ref{hello.08}) gives
\begin{equation}\label{hello.09}
 (\rho-k)^p |G(\rho)|\leq
C\|f\|_{L^{qs}}^p|G(k)|^{\frac{p}{qr}}.
\end{equation}
Thus, for any $\rho>k$, we have
$$
|G(\rho)|\le
\left(\frac{C\|f\|_{L^{qs}}}{\rho-k}\right)^p|G(k)|^{\frac{p}{qr}}.
$$
Note that $q=\frac{p}{p-1}$ and $s=\frac{r}{r-1}$. Thus,
$$
|G(\rho)|\le
\left(\frac{C\|f\|_{L^{\frac{pr}{(p-1)(r-1)}}}}{\rho-k}\right)^p|G(k)|^{\frac{p-1}{r}}.
$$
Since, by assumption, $p>2$ and $1\le r< p-1$, then we have
$\frac{p-1}{r}>1$. Thus, with $\phi(s)=|G(s)|$, it follows from the
De Giorgi's Lemma \ref{FEME:lem1} that
\begin{equation}\label{bravo}
|G(d+k_*)|=0,
\end{equation}
where
$$
d= C 2^{\frac{p-1}{p-1-r}}
|\Omega|^{\frac{p-1-r}{pr}}\|f\|_{L^{\frac{pr}{(p-1)(r-1)}}}.
$$
The equation (\ref{bravo}) implies that $u_h\le d+k_*$ on $\Omega$,
which can be rewritten as
$$
\sup_{\Omega} u_h\le k_* +
C 2^{\frac{p-1}{p-1-r}}
|\Omega|^{\frac{p-1-r}{pr}}\|f\|_{L^{\frac{pr}{(p-1)(r-1)}}}.
$$
This completes the proof.
\end{proof}

\bigskip
The rest of this section will establish another discrete maximum
principle for the underlying quasi-linear second order equation when
$f\le 0$. The result can be stated as follows.

\begin{theorem}\label{thm5.3}
Let $u_h\in S_h$ be the $P1$-conforming finite element approximation
of (\ref{FEME:eq1}) and (\ref{bc-diri}) arising from the formula
(\ref{weakform-fem}). Let $f\le 0$ be any locally integrable
function, and the ellipticity (\ref{ellipticity}) and the
boundedness (\ref{boundedness}) are satisfied. Assume that the {\em
Assumption} \ref{assumptionA} holds true. Then, we have
\begin{equation}\label{dmp-case-nof}
\sup_{x\in \Omega} u_h(x) \le
\left\{
\begin{array}{ll}
\sup_{x\in\partial\Omega}
\max(I_hg(x), 0),&\quad \mbox{ if \ $c\ge 0$, }\\
\sup_{x\in\partial\Omega} I_hg(x), &\quad \mbox{ if \ $c\equiv 0$,}
\end{array}
\right.
\end{equation}
provided that the meshsize $h$ is sufficiently small such that
\begin{equation}\label{hnucondition}
h\nu < 1.
\end{equation}
\end{theorem}
\medskip

\begin{proof} Assume that the maximum principle (\ref{dmp-case-nof}) does not hold true.
We show that such an assumption shall lead to a contradiction. To
this end, using the notation as given in (\ref{kzero}), we see that
$k_*<k_M \equiv \sup_{x\in \Omega} u_h(x)$. Let $k_\#$ be the
largest nodal value of $u_h$ (including the nodal points on the
boundary of $\Omega$) that is smaller than $k_M$. Let $k$ be any
real number such that $k_\#\le k < k_M$. Let $\varphi=(u_h-k)_+\in
S_h$ be the positive part of $u_h-k$ at nodal points. Since $k\ge
k_\#\ge k_*$ and $k_*$ is no smaller than the maximum value of the
finite element solution $u_h$ on $\partial\Omega$, then
(\ref{ukpzero}) holds true. By choosing $v=\varphi$ in
(\ref{weakform-fem}), we obtain from (\ref{technical.01}) and the
assumption of $f\le 0$ that
\begin{eqnarray}\nonumber
0\ge F(\varphi)&=&\mathfrak{Q}(u_h;u_h,\varphi) =\mathfrak{Q}(u_h;u_h, (u_h-k)_+)  \\
&\ge& \mathfrak{Q}(u_h;(u_h-k)_+, (u_h-k)_+) + \mathfrak{Q}(u_h;(u_h-k)_-, (u_h-k)_+).
\label{newmp.18}
\end{eqnarray}
Now using the Assumption \ref{assumptionA} and the notation of $\varphi = (u_h-k)_+$ we obtain
$$
\mathfrak{Q}(u_h;\varphi, \varphi) \le 0,
$$
which leads to
\begin{eqnarray}\label{newmp.19}
(a\nabla \varphi, \nabla\varphi) + (\bbb\cdot\nabla\varphi, \varphi)
+ (c\varphi,\varphi)\le 0.
\end{eqnarray}
Thus, we have from the ellipticity (\ref{ellipticity}), the
boundedness (\ref{boundedness}), and the condition of $c\ge 0$ that
\begin{eqnarray}\nonumber
\lambda \|\nabla\varphi\|_{L^2}^2 &\leq& (a\nabla \varphi,
\nabla\varphi)\\
&\le& \left| (\bbb\cdot\nabla\varphi, \varphi) \right|\nonumber\\
&\le& \lambda \nu \|\nabla\varphi\|_{L^2}\
\|\varphi\|_{L^2(D_k)},\label{newmp.850}
\end{eqnarray}
where $D_k$ is the subset of $\Omega$ on which $\nabla \varphi\neq
0$. Note that $D_k$ is a collection of triangular or tetrahedral
elements. It follows from the last inequality that
\begin{equation}\label{newmp.851}
\|\nabla\varphi\|_{L^2} \leq \nu \|\varphi\|_{L^2(D_k)}.
\end{equation}
The inequality (\ref{newmp.851}) can be rewritten by using element
integrals as follows
\begin{equation}\label{newmp.861}
\sum_{T\in D_k} \int_T |\nabla\varphi |^2 dT \leq \nu^2 \sum_{T\in
D_k} \int_T |\varphi|^2 dT.
\end{equation}
On each $T\subset D_k$, since $\nabla\varphi\neq 0$, then $\varphi$
is not a constant on $T$. Therefore, the selection of $k$ implies
that $\varphi=0$ at one of the vertices of $T$. Assume that
$\varphi(A)=0$ with $A$ being a vertex point of $T$. Then, we have
from $\varphi(x)=(x-A)\cdot\nabla\varphi$ that
$$
\int_T|\varphi|^2 dT \leq h_T^2 \int_T |\nabla \varphi|^2 dT,
$$
where $h_T$ is the diameter of the element $T$. Substituting the
above into (\ref{newmp.861}) we obtain
\begin{equation}\label{newmp.871}
\sum_{T\in D_k} \int_T |\nabla\varphi |^2 dT \leq \nu^2 \sum_{T\in
D_k} h_T^2 \int_T |\nabla\varphi|^2 dT\leq \nu^2 h^2 \sum_{T\in D_k}
\int_T |\nabla\varphi|^2 dT,
\end{equation}
which leads to
$$
1\leq h\nu.
$$
The above inequality is an obvious contradiction to the assumption
of $h \nu < 1$ as given in (\ref{hnucondition}). This completes the
proof.
\end{proof}

\section{Nodal Basis and Geometry of Finite Elements}

On each triangle or tetrahedron $T\in {\cal T}_h$, the finite
element function $v\in S_h$ is a linear function and can be
represented by local shape functions $\ell_i=\ell_i(x)$ defined as
follows: (1) $\ell_i$ is linear on $T$, (2)
$\ell_i(A(j))=\delta_{ij}$ where $\delta_{ij}$ is the usual
Kronecker symbol (see Fig. \ref{fig:triangle}). The local
representative property asserts that
\begin{equation}\label{local-representation}
v(x) = \sum_{i=1}^{d+1} v(A(i)) \ell_i(x), \qquad\forall x\in T.
\end{equation}

\begin{figure}
\begin{center}
\begin{tikzpicture}[rotate=233]
    \path (0,0) coordinate (A3);
    \path (3,0) coordinate (A1);
    \path (3.4,-0.2) coordinate (A11);
    \path (0,4) coordinate (A2);
    \path (-0.2,4.4) coordinate (A21);
    \path (1.0,1.0) coordinate (center);
    \path (1.5,0) coordinate (A1half);
    \path (0,2) coordinate (A2half);
    \path (1.5,2) coordinate (A3half);
    \path (2.0,0.5) coordinate (a);
    \path (A1half) ++(0,-0.6) coordinate (A1To);
    \path (A2half) ++(-0.6,0) coordinate (A2To);
    \path (A3half) ++(38:0.6cm) coordinate (A3To);
    \draw (A3) -- (A1) -- (A2) -- (A3);
    \filldraw[black] (A1) circle(0.1);
    \filldraw[black] (A2) circle(0.1);
    \filldraw[black] (A3) circle(0.1);
    \draw node[above] at (A3)(2.625, 0.5) {A(3)};
    \draw node[below] at (A11) {A(1)};
    \draw node[below] at (A21) {A(2)};
    \draw node at (center) {T};

    \draw[->,thick] (A1half) -- (A1To) node[above]{$\mathbf{n}(2)$};
    \draw[->,thick] (A2half) -- (A2To) node[right]{$\mathbf{n}(1)$};
    \draw[->,thick] (A3half) -- (A3To) node[right]{$\mathbf{n}(3)$};
    \draw (2.625, 0.5) arc(135:180:0.725);
    \draw node at (a) {$\alpha_{23}$};
\end{tikzpicture}
\end{center}
\caption{A triangular element with acute angles}
\label{fig:triangle}
\end{figure}
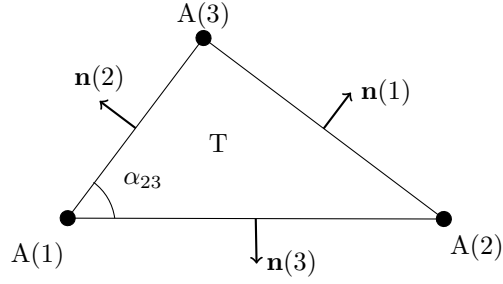

Note that the gradient of a function $\psi=\psi(x)$ is a vector
along which the function $\psi$ increases the most. Thus, the
gradient of the shape function $\ell_i$ would be parallel to the
outward normal direction of the edge/face opposite to the vertex
$A(i)$; i.e.,
$$
\nabla\ell_i = \alpha_i \bn(i),
$$
where $\bn(i)$ represents the outward normal direction to the edge/face opposite
to the vertex $A(i)$ (see Fig. \ref{fig:triangle} and Fig. \ref{fig:tetrahedron}).
Denote by $\|\xi\|$ the $\ell^2$-length of any vector $\xi\in \mathbb{R}^d$.
It follows that
$$
\alpha_i=-\|\nabla \ell_i\|.
$$
Thus, we have
\begin{equation}\label{direction}
\nabla\ell_i = - \|\nabla \ell_i\| \bn(i).
\end{equation}

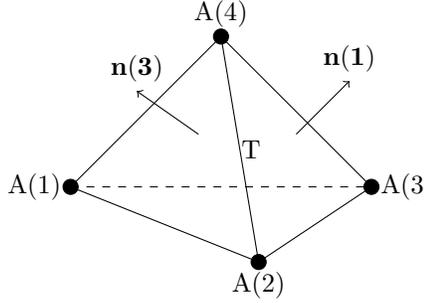
\begin{figure}
\begin{center}
\begin{tikzpicture}
\coordinate (A1) at (-2,0); \coordinate (A2) at (0.5,-1);
\coordinate (A3) at (2,0); \coordinate (A4) at (0,2); \coordinate
(center) at (0.4,0.5); \coordinate (n3) at (-0.3,0.7); \path (n3)
++(145:1) coordinate (n3end); \coordinate (n4) at (1.0,0.7); \path
(n4) ++(45:1) coordinate (n4end); \draw node[left] at (A1) {A(1)};
\draw node[below] at (A2) {A(2)}; \draw node[right] at (A3) {A(3)};
\draw node[above] at (A4) {A(4)}; \draw node at (center) {T}; \draw
(A1)--(A2)--(A3)--(A4)--cycle; \draw (A2)--(A4); \draw[dashed]
(A1)--(A3); \draw[->] (n3)--(n3end) node[above]{$\mathbf{n(3)}$};
\draw[->] (n4)--(n4end) node[above]{$\mathbf{n(1)}$};
    \filldraw[black] (A1) circle(0.1);
    \filldraw[black] (A2) circle(0.1);
    \filldraw[black] (A3) circle(0.1);
    \filldraw[black] (A4) circle(0.1);
\end{tikzpicture}
\caption{A tetrahedron with acute interior dihedral angles}
\label{fig:tetrahedron}
\end{center}
\end{figure}

The angles of the triangle $\Delta A(1)A(2)A(3)$ (see Fig.
\ref{fig:triangle}) can be characterized by using the outward normal
directions $\bn(i)$. For example, the angle $\alpha_{23}$ is related
to the angle of the two normal vectors $\bn(2)$ and $\bn(3)$ as
follows:
$$
\alpha_{23} = \pi -  \angle(\bn(2),\bn(3)),
$$
where $\angle(\bn(2),\bn(3))$ stands for the angle between $\bn(2)$
and $\bn(3)$. Likewise, for the tetrahedron $T$ as depicted in Fig.
\ref{fig:tetrahedron}, the interior angle between the two planes
$P(A(1), A(2), A(4))$ and $P(A(2), A(3), A(4))$ can be defined as
$$
\theta = \pi-\angle(\bn(1), \bn(3)).
$$
The angle $\theta$ is known as an interior dihedral angle.  The
definition of other five interior dihedral angles for $T$ can be
defined similarly. For simplicity, we introduce the following
notation:
\begin{equation}\label{angles}
\alpha_{ij}:=\pi - \angle(\bn(i), \bn(j)).
\end{equation}
It follows from (\ref{direction}) that
\begin{equation}\label{angles-ell}
\alpha_{ij}=\pi - \angle(\nabla\ell_i, \nabla\ell_j).
\end{equation}
The triangle $T$ is called non-obtuse if all the angles satisfy
$\alpha_{ij}\le \pi/2$. It is said to be acute if
$\alpha_{ij}<\pi/2$. Likewise, a tetrahedron $T$ is called acute if
each of its six interior dihedral angles is less than $\pi/2$ in
radian; $T$ is said to be non-obtuse if all six interior dihedral
angles are no more than $\pi/2$ in radian. For the purpose of the
maximum principles for finite element approximations, we introduce
the following concept.
\medskip
\begin{definition}
The finite element partition $\T_h$ is called
$\mathcal{O}(h^\alpha)$-acute if there exists a parameter $\gamma>0$
such that for each element $T\in \T_h$ we have $\alpha_{ij}\le
\frac{\pi}{2}-\gamma h^\alpha$, where $\alpha\ge0$ and $h$ is the
meshsize of $\T_h$.
\end{definition}

\section{Verification of the Key Assumption for DMP}\label{Section-Technical}
Recall that the validity of DMPs as shown in Theorems
\ref{thm5.2} and \ref{thm5.3} is based on the Assumption \ref{assumptionA} which states
that
\begin{eqnarray}\label{remark.01new}
\mathfrak{Q}(u_h;(u_h-k)_-,(u_h-k)_+) \ge 0
\end{eqnarray}
for all $k\ge k_*$. The goal of this section is to verify the above assumption under certain
conditions for the finite element partition $\T_h$.

\subsection{An Element-Based Approach}

By an element-wise approach, we mean a representation of the form
$\mathfrak{Q}(u_h;(u_h-k)_-,(u_h-k)_+)$ as integrals over each element $T\in\T_h$.
To verify the assumption (\ref{remark.01new}), we shall explore conditions that make each element integral be
non-negative. To this end, on each element $T\in \T_h$,
we use the local shape functions $\ell_j$ to represent both $(u_h-k)_-$ and $(u_h-k)_+$
as follows
\begin{eqnarray*}
(u_h-k)_-(x)&=&\sum_{i=1}^{d+1} (u_h(A(i))-k)_-\ell_i(x),\\
(u_h-k)_+(x)&=&\sum_{j=1}^{d+1} (u_h(A(j))-k)_+\ell_j(x).
\end{eqnarray*}
Denote by $\varphi= (u_h-k)_+$ and $\psi=(u_h-k)_-$. It follows that
\begin{eqnarray*}\label{hello.99}
\mathfrak{Q}(u_h;(u_h-k)_-,(u_h-k)_+) &=& \mathfrak{Q}(u_h;\psi,\varphi) \\
&=&\sum_{T\in \T_h}\left\{(a\nabla\psi, \nabla\varphi)_T+(\bbb\cdot\nabla\psi, \varphi)_T +(c\psi, \varphi)_T  \right\}
\end{eqnarray*}
On each element $T$, we have
\begin{eqnarray}\label{mainterm.05v}
&&(a\nabla\psi, \nabla\varphi)_T+(\bbb\cdot\nabla\psi, \varphi)_T +(c\psi, \varphi)_T\nonumber \\
=&&\sum_{i,j=1}^{d+1} (u_h(A(i))-k)_-\ (u_h(A(j))-k)_+ \int_T \left\{ a
\nabla\ell_i \cdot \nabla\ell_j + \bbb\cdot(\nabla\ell_i) \ell_j + c
\ell_i\ell_j \right\}dx.\nonumber
\end{eqnarray}
Using the angle relation (\ref{angles-ell}) we obtain
\begin{eqnarray*}
\nabla\ell_i \cdot \nabla\ell_j &=& \|\nabla\ell_i\|\ \|\nabla\ell_j\| \cos(\angle(\nabla\ell_i,\nabla\ell_j))\\
&=& \|\nabla\ell_i\|\ \|\nabla\ell_j\| \cos(\pi-\alpha_{ij})\\
&=& - \|\nabla\ell_i\|\ \|\nabla\ell_j\| \cos(\alpha_{ij}).
\end{eqnarray*}
Thus, it follows from the boundedness (\ref{boundedness}) that
\begin{eqnarray*}
&&-\int_T \left\{a \nabla\ell_i \cdot \nabla\ell_j +
\bbb\cdot(\nabla\ell_i) \ell_j + c \ell_i\ell_j
\right\}dx \\
&&=\int_T \left\{a \|\nabla\ell_i\|\
\|\nabla\ell_j\|\cos(\alpha_{ij}) - \bbb\cdot(\nabla\ell_i) \ell_j -
c \ell_i\ell_j
\right\}dx\\
&&\ge \int_T \left\{a \|\nabla\ell_i\|\  \|\nabla\ell_j\|\cos(\alpha_{ij}) -\|\bbb\|\ \|\nabla\ell_i\| - |c|\right\}dx\\
&&\ge \int_T \left\{a \|\nabla\ell_i\|\
\|\nabla\ell_j\|\cos(\alpha_{ij}) -\lambda\nu\
\|\nabla\ell_i\| - \lambda \nu \right\}dx.\\
\end{eqnarray*}
Assume that the element $T$ is non-obtuse (i.e., $0\le\alpha_{ij}\le
\pi/2$). Then we have from the above inequality and the ellipticity
(\ref{ellipticity}) that
\begin{eqnarray*}
&&-\int_T \left\{a \nabla\ell_i \cdot \nabla\ell_j +
\bbb\cdot(\nabla\ell_i) \ell_j + c \ell_i\ell_j
\right\}dx \\
&&\ge \lambda \int_T \left\{\|\nabla\ell_i\|\
\|\nabla\ell_j\|\cos(\alpha_{ij}) -\nu(\|\nabla\ell_i\| +1)\right\}dx.\\
\end{eqnarray*}
Next, we see from Taylor expansion, for $\alpha_{ij}\in
[\rho_0,\pi/2]$ with $\rho_0>0$ being a fixed angle, there is a
constant $\gamma^*>0$ such that
$$
\cos(\alpha_{ij})\ge \gamma^*\left(\frac{\pi}{2}-\alpha_{ij}\right).
$$
Observe that both $\|\nabla\ell_i\|$ and $\|\nabla\ell_j\|$ are of
size $\mathcal{O}(h_T^{-1})$ where $h_T$ is the size of $T$. Thus,
with $|T|$ being the measure of $T$, we have
\begin{eqnarray*}
&&-\int_T \left\{a \nabla\ell_i \cdot \nabla\ell_j +
\bbb\cdot(\nabla\ell_i) \ell_j + c \ell_i\ell_j
\right\}dx \\
&&\ge \lambda \gamma^*\int_T \left\{\|\nabla\ell_i\|\  \|\nabla\ell_j\|\left(\pi/2 -\alpha_{ij}\right) -\nu(\|\nabla\ell_i\| +1)\right\}dx\\
&&\ge \lambda^* \|\nabla\ell_i\|\  \|\nabla\ell_j\|\ |T|
\end{eqnarray*}
for some $\lambda^*>0$ when the size of $T$ is sufficiently small
and $\pi/2-\alpha_{ij}\ge \gamma h$ for a large, but fixed constant
$\gamma$. In the case of $\bbb=0$, the angle requirement can be
weakened to $\pi/2-\alpha_{ij}\ge \gamma h^2$. The very same argument holds true if $u_h$ is replaced by
any finite element function $v\in S_h$. The result can be
summarized into a lemma as follows.

\medskip
\begin{lemma}\label{technical-lemma2}
Let $v\in S_h$ be any finite element function and $k$ any real
number. Assume that the ellipticity (\ref{ellipticity}) and the
boundedness (\ref{boundedness}) hold true. Assume also that the
partition $\T_h$ is $\mathcal{O}(h^\alpha)$-acute. Then, the
following results hold true:
\medskip
\begin{enumerate}
\item[(i)] For general $\bbb$ and $c\ge 0$, with $\alpha=1$, we have
\begin{eqnarray}\label{technical.08}
&&\mathfrak{Q}(v;(v-k)_-,(v-k)_+)\ge \\
&& \lambda^* \sum_{T\in \T_h} \sum_{i\neq j} |(v(A(i))-k)_-|\
|(v(A(j))-k)_+|\
 \|\nabla\ell_i\|\  \|\nabla\ell_j\|\ |T|,\nonumber
\end{eqnarray}
provided that the meshsize $h$ for the partition $\T_h$ is sufficiently small. Here $\lambda^*$
is a positive number smaller than $\lambda$ and $|T|$ stands for the area or volume of the element $T$.
\item[(ii)] For the case $\bbb=0$ and $c\ge 0$, with $\alpha=2$, we have
\begin{eqnarray}\label{technical.08-01}
&&\mathfrak{Q}(v;(v-k)_-,(v-k)_+)\ge \\
&&\lambda^* \sum_{T\in \T_h} \sum_{i\neq j} |(v(A(i))-k)_-|\
|(v(A(j))-k)_+|\ \|\nabla\ell_i\|\  \|\nabla\ell_j\|\ |T|,\nonumber
\end{eqnarray}
provided that $h$ is sufficiently small.
\item[(iii)] For the case of $\bbb=0$ and $c=0$, we have
\begin{eqnarray}\label{technical.09}
&& \quad \mathfrak{Q}(v;(v-k)_-,(v-k)_+)\ge\\
&& \lambda \sum_{T\in \T_h} \sum_{i\neq j} |(v(A(i))-k)_-|\
|(v(A(j))-k)_+|\
 \|\nabla\ell_i\|\  \|\nabla\ell_j\|\ \cos(\alpha_{ij}) |T|,\nonumber
\end{eqnarray}
as long as each $T\in \T_h$ is non-obtuse.
\end{enumerate}
In other words, the Assumption \ref{assumptionA} is satisfied if the finite element
partition $\T_h$ satisfies certain angle conditions.
\end{lemma}

\subsection{An Edge-Based Approach}

\begin{figure}[h]
\setlength{\unitlength}{0.14in} 
\centering 
\begin{picture}(32,15) 
\put(8,7){\line(1,1){6}} \put(14,13){\line(3,-1){12.45}}
\put(8,7){\line(3,-2){9}} \put(17,1){\line(6,5){9.5}}
\put(14,13){\line(1,-4){3}} \put(12.3,7) {$T_1$} \put(18.3,7)
{$T_2$} \put(17,0) {$A$} \put(13.5,13.5) {$B$} \put(7,6.5) {$C$}
\put(8.6,6.8) {$\alpha$} \put(26.8,8.5) {$D$}
\qbezier(9.2,8.2)(10.0,7.2)(9.4,6.1)
\qbezier(24.5,9.5)(24.0,8.4)(24.7,7.5)
\put(25,8.3){$\beta$}
\end{picture}
\caption{An interior edge shared by two elements $T_1$ and $T_2$.} 
\label{fig-int}
\end{figure}
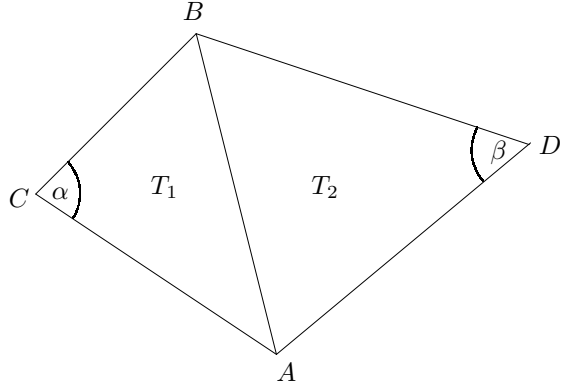

By an edge-wise approach, we mean a representation of the form
$\mathfrak{Q}(u_h;(u_h-k)_-,(u_h-k)_+)$ as integrals over macro-elements that share a common edge.
To verify the assumption (\ref{remark.01new}), we shall explore conditions that make each integral on macro-elements be
non-negative. To this end, we use the notation $\varphi=(u_h-k)_+$ and
$\psi=(u_h-k)_-$ to arrive at
\begin{eqnarray*}
&&\mathfrak{Q}(u_h;(u_h-k)_-,(u_h-k)_+) = \mathfrak{Q}(u_h;\psi,\varphi) \\
=&&\sum_{T\in\T_h}\sum_{i,j=1}^{d+1} \psi(A(i))\
\varphi(A(j)) \int_T \left\{ a \nabla\ell_i \cdot \nabla\ell_j +
\bbb\cdot(\nabla\ell_i) \ell_j + c \ell_i\ell_j \right\}dx\\
=&&\sum_{e_{mn}\in \mathcal{E}_h^0} \psi(A_m)\
\varphi(A_n)\sum_{s=1}^2\int_{T_s}\left\{a \nabla\ell_m^{(s)}
\cdot \nabla\ell_n^{(s)} + \bbb\cdot(\nabla\ell_m^{(s)})
\ell_n^{(s)} + c \ell_m^{(s)}\ell_n^{(s)}\right\}dx,
\end{eqnarray*}
where $\mathcal{E}_h^0$ denotes the set of all interior edges, $A_m$
and $A_n$ are two end points of the edge $e_{mn}$, $T_1$ and $T_2$
share $e_{mn}$ as a common edge. In Fig. \ref{fig-int}, one may
identify $A_m$ with $A$, and $A_n$ with $B$. Here $\ell_m^{(s)}$ is
the shape function on the element $T_s$ associated with the vertex
point $A_m$. Thus, the validity of various DMPs can be derived if
the following holds true
\begin{equation}\label{edge-info}
\sum_{s=1}^2\int_{T_s}\left\{a \nabla\ell_m^{(s)} \cdot
\nabla\ell_n^{(s)} + \bbb\cdot(\nabla\ell_m^{(s)}) \ell_n^{(s)} + c
\ell_m^{(s)}\ell_n^{(s)}\right\}dx\le 0.
\end{equation}
In the case of Poisson problem, one has $a\equiv 1,\ \bbb\equiv 0,$
and $c\equiv 0$. Thus, it suffices to have
\begin{equation}\label{edge-info-01}
\sum_{s=1}^2\int_{T_s}\nabla\ell_m^{(s)} \cdot
\nabla\ell_n^{(s)}dx\le 0.
\end{equation}
It was known that (see for example \cite{scott})
$$
\int_{T_1}\nabla\ell_m^{(1)} \cdot \nabla\ell_n^{(1)}dx =
-\frac{\cot(\alpha)}{2},
$$
and
$$
\int_{T_2}\nabla\ell_m^{(2)} \cdot \nabla\ell_n^{(2)}dx =
-\frac{\cot(\beta)}{2}.
$$
It follows that
\begin{eqnarray*}
\sum_{s=1}^2\int_{T_s}\nabla\ell_m^{(s)} \cdot \nabla\ell_n^{(s)}dx
&=& -\frac{\cot(\alpha)}{2} - \frac{\cot(\beta)}{2}\\
&=& -\frac{\sin(\alpha+\beta)}{2\sin\alpha \ \sin\beta},
\end{eqnarray*}
and (\ref{edge-info-01}) holds true if and only if $\alpha+\beta \le
\pi$.

A similar, but more complicated, analysis can be conducted for
tetrahedral elements; this is left to readers with interest and
curiosity on DMPs for Poisson problems in 3D.

\section*{Acknowledgment}
The authors thank Professor Xiu Ye for helpful discussions and proof
reading of the manuscript. The authors also thank the anonymous
referees for suggesting a re-organization of the technical
presentation by making (\ref{hypothesis}) as a general assumption.

\end{document}